\documentclass[11pt,a4paper]{article}

\usepackage[T1]{fontenc}
\usepackage{lmodern}
\usepackage{amsmath,amssymb,amsthm,mathtools}
\usepackage{microtype}
\usepackage[margin=30mm]{geometry}
\usepackage[colorlinks=true,linkcolor=blue,citecolor=blue,urlcolor=blue]{hyperref}
\allowdisplaybreaks[2]
\numberwithin{equation}{section}

\newtheorem{theorem}{Theorem}[section]

\newtheorem{lemma}[theorem]{Lemma}
\newtheorem{corollary}[theorem]{Corollary}
\theoremstyle{remark}
\newtheorem{remark}[theorem]{Remark}

\newcommand{\Stir}[2]{\left\{\begin{matrix}#1\\#2\end{matrix}\right\}}
\newcommand{\Li}{\operatorname{Li}}

\usepackage[OT2, OT1]{fontenc}
\def\sha{{\ifmmode
\textup{\fontencoding{OT2}\selectfont sh}
\else
\fontencoding{OT2}\selectfont sh
\fi}}

\title{Explicit Formulas, Duality, and Periodicity\\
for Multi-Indexed Poly-Bernoulli Numbers}
\author{Yuna Baba, Tomoko Kikuchi and Maki Nakasuji}
\date{}

\begin{document}
\maketitle

\begin{abstract}
We study multi-indexed poly-Bernoulli numbers defined through multiple
polylogarithms of $\sha$-type.  The central result is an explicit
formula in terms of Stirling numbers of the second kind and multinomial
coefficients, valid in arbitrary depth and for arbitrary integral upper
indices.  We present a direct coefficient-extraction proof of the formula
and derive the duality relation for negative upper indices from a
symmetric multivariate exponential generating function.  We then use
these formulas to establish congruence periodicity in each lower index
modulo an arbitrary prime power, extending the double-indexed result of
Baba, Nakasuji, and Sakata.  A second application gives a congruence
for the sum over one complete period in the first lower index.  
\end{abstract}


\section{Introduction}

The classical Bernoulli numbers, with $B_1=1/2$, are
defined by the following power series expansions: 
\begin{equation}\label{eq:classical}
 \frac{te^t}{e^t-1}=\sum_{m=0}^{\infty}B_m\frac{t^m}{m!}
\end{equation}
for $m \in \mathbb{Z}_{\ge 0}$. 
As a generalization of $B_m$, Kaneko \cite{Kaneko} introduced poly-Bernoulli numbers by
\begin{equation}\label{eq:poly}
 \frac{\operatorname{Li}_{k}(1-e^{-t})}{1-e^{-t}}
 =\sum_{m=0}^{\infty}{\mathbb B}_m^{(k)}\frac{t^m}{m!},
 \qquad
 \operatorname{Li}_{k}(z)=\sum_{m=1}^{\infty}\frac{z^m}{m^k}.
\end{equation}
Their explicit formula, expressed using Stirling numbers of the second
kind, is
\begin{equation}\label{eq:kaneko-explicit}
 {\mathbb B}_m^{(k)}=(-1)^m\sum_{\ell=0}^{m}
 \frac{(-1)^\ell\ell!}{(\ell+1)^k}\Stir{m}{\ell}.
\end{equation}
In particular, negative upper indices lead to integer-valued numbers
with a duality property.  Related multiple polylogarithms and
poly-Bernoulli numbers have been investigated in
\cite{ArakawaKaneko1999,KanekoTsumura2018}.

Kaneko and Tsumura \cite{KanekoTsumura2018} introduced multi-indexed poly-Bernoulli numbers (defined by the generating function \eqref{eq:def}) using multiple polylogarithms of $\sha$-type and established their negative-index duality, generalizing Kaneko's one-variable identity.
Their duality theorem is stated as follows.
\begin{theorem}[\cite{KanekoTsumura2018}, Theorem~5.4]
\label{thm:KT-duality}
For $m_1,\ldots,m_r,k_1,\ldots,k_r\in\mathbb Z_{\geq0}$, we have
\begin{equation}\label{eq:KT-duality}
 {\mathbb B}_{m_1,\ldots,m_r}^{(-k_1,\ldots,-k_r)}
 ={\mathbb B}_{k_1,\ldots,k_r}^{(-m_1,\ldots,-m_r)}.
\end{equation}
\end{theorem}

\begin{remark}
Kaneko and Tsumura \cite{KanekoTsumura2018} proved
Theorem~\ref{thm:KT-duality} by studying a multivariable
$\eta$-function defined by
\begin{equation}\label{eq:eta-def}
\begin{aligned}
&\eta(-k_1,\ldots,-k_r;s_1,\ldots,s_r)\\
&\quad=
\frac{1}{\prod_{j=1}^r\Gamma(s_j)}
\int_0^\infty\cdots\int_0^\infty
\frac{
\Li^{\sha}_{-k_1,\ldots,-k_r}
(1-e^{T_1},\ldots,1-e^{T_r})
}{
\prod_{j=1}^r(1-e^{T_j})
}
\prod_{j=1}^r t_j^{s_j-1}\,dt_j,
\end{aligned}
\end{equation}
for $k_1,\ldots,k_r\in\mathbb Z_{\geq 0}$ and
$\operatorname{Re}(s_j)>0$ $(1\leq j\leq r)$,
where $T_i=t_i+\cdots+t_r$ (see \eqref{eq:T}).
This function admits an analytic continuation to an
entire function of $s_1,\ldots,s_r$ and satisfies
\begin{equation}\label{eq:eta-special}
\eta(-k_1,\ldots,-k_r;-m_1,\ldots,-m_r)
=
{\mathbb B}_{m_1,\ldots,m_r}^{(-k_1,\ldots,-k_r)}
\end{equation}
for nonnegative integers $m_1,\ldots,m_r$.
Kaneko and Tsumura established the symmetry
\begin{equation}\label{eq:KT-eta}
\eta(-k_1,\ldots,-k_r;-m_1,\ldots,-m_r)
=
\eta(-m_1,\ldots,-m_r;-k_1,\ldots,-k_r),
\end{equation}
which, together with \eqref{eq:eta-special},
yields Theorem~\ref{thm:KT-duality}.

Subsequently, Yamamoto \cite{Yamamoto}
extended the $\eta$-function to a function of
$2r$ complex variables and established the symmetry
\begin{equation}\label{eq:Yamamoto}
\eta(s'_1,\ldots,s'_r;s_1,\ldots,s_r)
=
\eta(s_1,\ldots,s_r;s'_1,\ldots,s'_r),
\end{equation}
where $s_i,s'_i\in\mathbb C$ for $1\leq i\leq r$.
This identity also gives the duality in
Theorem~\ref{thm:KT-duality} by specialization.
We also note that Komori and Tsumura
\cite{KomoriTsumura} obtained related results
in the one-variable case and its generalizations.
\end{remark}

On the other hand, Baba, Nakasuji, and Sakata
 \cite{BabaNakasujiSakata2025} obtained an explicit formula for double-indexed poly-Bernoulli numbers $\mathbb{B}_{m_1,m_2}^{(k_1,k_2)}$ in terms of the Stirling numbers of the second kind, as an extension of \eqref{eq:kaneko-explicit}. 
 We recall their result in the notation used here.

\begin{theorem}[\cite{BabaNakasujiSakata2025}, Theorem~9]
\label{thm:BNS-double}
For $k_1,k_2\in\mathbb Z$ and $m_1,m_2\in\mathbb Z_{\geq0}$,
\begin{equation}\label{eq:BNS-double}
\begin{aligned}
{\mathbb B}_{m_1,m_2}^{(k_1,k_2)}
 ={}&(-1)^{m_1+m_2}
 \sum_{\ell_1,\ell_2=0}^{m_1+m_2}
 \frac{(-1)^{\ell_1+\ell_2}\ell_1!\ell_2!}
 { (\ell_1+1)^{k_1}(\ell_1+\ell_2+2)^{k_2}}\\
 &\quad\times\sum_{n=0}^{m_2}
 \Stir{m_1+n}{\ell_1}\Stir{m_2-n}{\ell_2}
 \binom{m_2}{n}.
\end{aligned}
\end{equation}
\end{theorem}
 
 They also established periodicity congruences in each of the two lower indices and a congruence for the sum over a complete period in the first lower index. These results motivate the study of explicit formulas and their applications to duality and periodicity for multi-indexed poly-Bernoulli numbers of arbitrary depth.

The purpose of this paper is to establish a general explicit formula
for multi-indexed poly-Bernoulli numbers in arbitrary depth and to
investigate its applications to duality and congruences.  These extend the corresponding
double-indexed results of \cite{BabaNakasujiSakata2025}.

\section{Preliminaries}\label{sec:prelim}

For $r\geq1$ and $s_1,\ldots,s_r\in\mathbb C$, the multiple
polylogarithm of $\sha$-type is
\begin{align}
\mathrm{Li}_{s_1, \ldots, s_r}^{\sha} (z_1, \ldots, z_r) = \sum_{0 < m_1 < \dots < m_r} \frac{z_1^{m_1} z_2^{m_2 - m_1} \dots z_r^{m_r - m_{r - 1}}}{m_1^{s_1} m_2^{s_2} \dots m_r^{s_r}} 
\label{eq:lix}
\end{align}
for $r \in \mathbb{N}, s_1, \ldots, s_r \in \mathbb{C}$ and $z_1, \ldots, z_r \in \mathbb{C}$ with $|z_j| \le 1 \, (1 \le j \le r)$  (cf. \cite{AK2}, \cite{Goncharov}).
The multi-indexed poly-Bernoulli numbers are defined by
\begin{equation}\label{eq:def}
\begin{aligned}
&\frac{\Li^{\sha}
_{k_1,\ldots,k_r}
 (1-e^{-(t_1+\cdots+t_r)},\ldots,1-e^{-t_r})}
 {\prod_{j=1}^r(1-e^{-(t_j+\cdots+t_r)})}
\\[-2pt]
&\hspace{26mm}=
 \sum_{m_1,\ldots,m_r\geq0}
 {\mathbb B}_{m_1,\ldots,m_r}^{(k_1,\ldots,k_r)}
 \frac{t_1^{m_1}\cdots t_r^{m_r}}{m_1!\cdots m_r!}.
\end{aligned}
\end{equation}
Here the upper indices can be arbitrary integers.  For $r=1$ this
reduces to \eqref{eq:poly}.  The defining identity is understood as a
formal power series identity.

We use the following two auxiliary notations throughout:
%\begin{equation}\label{eq:T}
\begin{equation}\label{eq:T}
 T_i=t_i+\cdots+t_r \qquad(1\leq i\leq r),
\end{equation}
%\end{equation}
and, for $\boldsymbol\ell=(\ell_1,\ldots,\ell_r)$,
%\begin{equation}\label{eq:Q}
$$
 Q_j(\boldsymbol\ell)=\ell_1+\cdots+\ell_j+j
 \qquad(1\leq j\leq r).
$$
%\end{equation}
In particular, $Q_j(\boldsymbol\ell)$ is a positive integer when all
$\ell_j$ are nonnegative.

Then from \eqref{eq:lix} and \eqref{eq:def}, we have
\begin{align}
\sum_{m_1,\ldots,m_r\geq0}
 {\mathbb B}_{m_1,\ldots,m_r}^{(k_1,\ldots,k_r)}
 \frac{t_1^{m_1}\cdots t_r^{m_r}}{m_1!\cdots m_r!}
&=
\frac{\Li^{\sha}_{k_1,\ldots,k_r}
 (1-e^{-T_1},1-e^{-T_{2}},\ldots,1-e^{-T_r})}
 {\prod_{j=1}^r(1-e^{-T_{j}})}\notag\\%\label{eq:reversed}\\
& =\sum_{\ell_1,\ldots,\ell_r\geq0}
 \frac{\prod_{j=1}^r(1-e^{-T_{j}})^{\ell_j}}
 {\prod_{j=1}^r Q_j(\boldsymbol\ell)^{k_j}}.\label{eq:geom}
\end{align}
Although the sum on the right is infinite, every coefficient receives
contributions from only finitely many tuples of nonnegative integers.

The Stirling numbers of the second kind, say ${m \brace \ell}$, are defined for any $m, \ell \in \mathbb{Z}$ by the following recurrence formula: 
$$
{0 \brace 0} = 1, {m \brace 0} = {0 \brace \ell} = 0 \, (m, \ell \ne 0), {m + 1 \brace \ell} = {m \brace \ell - 1} + \ell{m \brace \ell}. 
$$
Here, ${m \brace \ell} = 0$ when $\ell > m$. 
It is known that the generating function for these numbers have the following formula. 
\begin{lemma} \label{Lem_Stirling}
For any $m \in \mathbb{Z}$ and $\ell \in \mathbb{Z}_{\ge 0}$, we have 
\begin{align*}
\sum_{m = \ell}^\infty {m \brace \ell} \frac{t^m}{m!}=\frac{(e^t - 1)^\ell}{\ell!}. 
\end{align*}
\end{lemma}

The following can be shown as an extension of Lemma \ref{Lem_Stirling}.
\begin{lemma}\label{lem:multi-stirling}
For $1\leq i\leq r$, any $m_i, \cdots, m_r\in \mathbb{Z}$ and $\ell \in \mathbb{Z}_{\ge 0}$, we have 
\begin{align*}
\sum_{m_i = 0}^\infty
\cdots
\sum_{m_r = 0}^\infty
 {m_i+\cdots+m_r\brace \ell} \frac{t_i^{m_i}\cdots t_{r}^{m_{r}}}{m_i!\cdots m_r!}=\frac{(e^{T_i} - 1)^\ell}{\ell!}. 
\end{align*}
Consequently,
\begin{equation}\label{eq:negative-stirling}
 (1-e^{-T_i})^\ell
 =(-1)^\ell\ell!
 \sum_{m_i,\ldots,m_r\geq0}
 (-1)^{m_i+\cdots+m_r}
 \Stir{m_i+\cdots+m_r}{\ell}
 \frac{t_i^{m_i}\cdots t_r^{m_r}}{m_i!\cdots m_r!}.
\end{equation}
\end{lemma}
\begin{proof}
\rm{This} can be obtained by a direct calculation.
Using the well-known relation between the Stirling numbers of the second kind and binomial coefficients,
$$
 {n \brace m}=\frac{(-1)^m}{m!}\sum_{k=0}^{m}(-1)^{k}
 \binom{m}{k}k^n,
$$
we have
\begin{align*}
&\sum_{m_i = 0}^\infty
\cdots
\sum_{m_r = 0}^\infty
 {m_i+\cdots+m_r \brace \ell} \frac{t_i^{m_i}\cdots t_r^{m_r}}{m_i!\cdots m_r!}\\
 &=
 \sum_{m_i = 0}^\infty
\cdots
\sum_{m_r = 0}^\infty
\left(
\frac{(-1)^{\ell}}{\ell !}\sum_{k=0}^{\ell}(-1)^k
\binom{\ell}{k}k^{m_i+\cdots + m_r}
\right)
\frac{t_i^{m_i}\cdots t_r^{m_r}}{m_i!\cdots m_r!}\\
&=
\frac{(-1)^{\ell}}{\ell !}
\sum_{k=0}^{\ell}
(-1)^k
\binom{\ell}{k}
\left(
\sum_{m_i=0}^{\infty}
\frac{(kt_i)^{m_i}}{m_i !}
\right)
\cdots
\left(
\sum_{m_r=0}^{\infty}
\frac{(kt_r)^{m_r}}{m_r !}
\right)\\
&=
\frac{(-1)^{\ell}}{\ell !}
\sum_{k=0}^{\ell}
(-1)^k
\binom{\ell}{k}
e^{(t_i+\cdots +t_r)k}
(-1)^{\ell-k}
=\frac{(e^{t_i+\cdots +t_r}-1)^{\ell}}{\ell !}.
\end{align*}
%Replace $t$ by $T_i$ in  Lemma \ref{Lem_Stirling} and apply the
%multinomial theorem to every power of $T_i$.  Replacing each
%$t_j$ by $-t_j$ and multiplying by $(-1)^\ell\ell!$ gives
%\eqref{eq:negative-stirling}.
\end{proof}

\section{The explicit formula}\label{sec:explicit}

We now establish an explicit formula for multi-indexed
poly-Bernoulli numbers in arbitrary depth.  It extends Kaneko's
one-variable formula \eqref{eq:kaneko-explicit} and the
double-indexed formula of Baba, Nakasuji, and Sakata
\cite{BabaNakasujiSakata2025}.  The proof is a direct coefficient
extraction from \eqref{eq:geom}.

\begin{theorem}\label{thm:explicit}
Let $r\geq1$, $k_1,\ldots,k_r\in\mathbb Z$, and
$m_1,\ldots,m_r\in\mathbb Z_{\geq0}$.  Then
\begin{equation}\label{eq:explicit}
\begin{aligned}
&{\mathbb B}_{m_1,\ldots,m_r}^{(k_1,\ldots,k_r)}
=(-1)^{m_1+\cdots+m_r}
\sum_{\substack{\ell_1,\ldots,\ell_r\geq0\\
\ell_1+\cdots+\ell_r\leq m_1+\cdots+m_r}}
\frac{(-1)^{\ell_1+\cdots+\ell_r}\ell_1!\cdots\ell_r!}
{\prod_{j=1}^r Q_j(\boldsymbol\ell)^{k_j}}
\\[-2pt]
&\quad\times
\sum_{\substack{n_{i,1}+\cdots+n_{i,i+1}=m_{i+1}\\
1\leq i\leq r-1;\ n_{i,j}\geq0}}
\left(\prod_{i=1}^{r-1}
\binom{m_{i+1}}{n_{i,1},\ldots,n_{i,i+1}}\right)
\prod_{j=1}^r
\Stir{m_j-\sum_{h=1}^{j-1}n_{j-1,h}
+\sum_{h=j}^{r-1}n_{h,j}}{\ell_j}.
\end{aligned}
\end{equation}
The sums over
$n_{i,j}$ are over nonnegative integers satisfying the displayed row
conditions.  For $r=1$, the sum over the indices $n_{i,j}$ is absent,
and the remaining Stirling factor is $\Stir{m_1}{\ell_1}$.
\end{theorem}

\begin{proof}
Starting with \eqref{eq:geom}, expand each factor in its numerator
by \eqref{eq:negative-stirling}.  We obtain
\begin{equation}\label{eq:extraction-start}
\begin{aligned}
&\sum_{m_1,\ldots,m_r\geq0}
 {\mathbb B}_{m_1,\ldots,m_r}^{(k_1,\ldots,k_r)}
 \frac{t_1^{m_1}\cdots t_r^{m_r}}{m_1!\cdots m_r!}
\\
&=\sum_{\ell_1,\ldots,\ell_r\geq0}
 \frac{(-1)^{\ell_1+\cdots+\ell_r}\ell_1!\cdots\ell_r!}
 {\prod_{j=1}^r Q_j(\boldsymbol\ell)^{k_j}}
\\[-2pt]
&\quad\times\prod_{j=1}^r
 \left(\sum_{a_j,\ldots,a_r\geq0}
 \Stir{a_j+\cdots+a_r}{\ell_j}
 \prod_{h=j}^r\frac{(-t_h)^{a_h}}{a_h!}
 \right).
\end{aligned}
\end{equation}
We compare the coefficients of $t_1^{m_1}\cdots t_r^{m_r}$.
The variable $t_1$ appears only in the factor corresponding to
$j=1$, so all $m_1$ copies of $t_1$ come from that factor.
For $2\leq h\leq r$, the variable $t_h$ appears in exactly the
factors indexed by $j=1,\ldots,h$.  Distribute its $m_h$ copies
among these factors, with $n_{h-1,j}$ copies coming from the factor
indexed by $j$.  Hence
\begin{equation}\label{eq:rows}
 n_{h-1,1}+\cdots+n_{h-1,h}=m_h
 \qquad(2\leq h\leq r).
\end{equation}
For any such distribution, the total degree taken from the factor
indexed by $j$ is
\begin{equation}\label{eq:stirling-upper}
 m_j-\sum_{h=1}^{j-1}n_{j-1,h}
       +\sum_{h=j}^{r-1}n_{h,j}.
\end{equation}
Indeed, the $t_j$ contribution from this factor is $m_j$ minus
the contributions already assigned to factors $1,\ldots,j-1$;
for $t_{j+1},\ldots,t_r$, it receives the terms
$n_{j,j},\ldots,n_{r-1,j}$.
For $j=1$, the first sum in \eqref{eq:stirling-upper} is empty,
giving $m_1+n_{1,1}+\cdots+n_{r-1,1}$.

Upon multiplying by $m_1!\cdots m_r!$, the factorial denominators
in \eqref{eq:extraction-start} yield the product of multinomial
coefficients
\[
 \prod_{i=1}^{r-1}
 \binom{m_{i+1}}{n_{i,1},\ldots,n_{i,i+1}}.
\]
All contributions of the prescribed total degree have the sign
$(-1)^{m_1+\cdots+m_r}$.  The factor indexed by $j$ contributes
the Stirling number whose upper argument is \eqref{eq:stirling-upper}.
Comparing coefficients therefore gives \eqref{eq:explicit},
initially with the outer sum extending over all $\ell_j\geq0$.

If the product of Stirling numbers in a summand does not vanish,
then
\[
 \ell_1+\cdots+\ell_r
 \leq\sum_{j=1}^r
 \left(m_j-\sum_{h=1}^{j-1}n_{j-1,h}
             +\sum_{h=j}^{r-1}n_{h,j}\right)
 =m_1+\cdots+m_r.
\]
The last equality follows from the row conditions
\eqref{eq:rows}, because each allocation occurs once with a
positive and once with a negative sign.
Thus the outer sum may be restricted as in \eqref{eq:explicit}.
\end{proof}

\begin{remark}\label{rem:small-depth}
For $r=1$, Theorem~\ref{thm:explicit} is exactly
\eqref{eq:kaneko-explicit}.  For $r=2$ it gives
\begin{equation}\label{eq:double}
\begin{aligned}
&{\mathbb B}_{m_1,m_2}^{(k_1,k_2)}
=(-1)^{m_1+m_2}
\sum_{\substack{\ell_1,\ell_2\geq0\\
\ell_1+\ell_2\leq m_1+m_2}}
\frac{(-1)^{\ell_1+\ell_2}\ell_1!\ell_2!}
{Q_1(\boldsymbol\ell)^{k_1}Q_2(\boldsymbol\ell)^{k_2}}
\\[-2pt]
&\hspace{18mm}\times\sum_{a=0}^{m_2}
\binom{m_2}{a}\Stir{m_1+a}{\ell_1}
\Stir{m_2-a}{\ell_2},
\end{aligned}
\end{equation}
recovering the formula of Baba, Nakasuji, and Sakata
\cite{BabaNakasujiSakata2025}.
For $r=3$, the inner sum in \eqref{eq:explicit} is
\begin{equation}\label{eq:triple}
\begin{aligned}
&\sum_{\substack{n_{1,1}+n_{1,2}=m_2\\
n_{2,1}+n_{2,2}+n_{2,3}=m_3}}
\binom{m_2}{n_{1,1},n_{1,2}}
\binom{m_3}{n_{2,1},n_{2,2},n_{2,3}}
\\
&\qquad\times
\Stir{m_1+n_{1,1}+n_{2,1}}{\ell_1}
\Stir{m_2-n_{1,1}+n_{2,2}}{\ell_2}
\Stir{m_3-n_{2,1}-n_{2,2}}{\ell_3}.
\end{aligned}
\end{equation}
Thus the arbitrary-depth formula also gives an explicit
triple-indexed expression without introducing additional
coefficient notation.
\end{remark}

\section{Duality as an application of the explicit formula}

Kaneko and Tsumura established duality by the study of a
multivariable $\eta$-function \cite{KanekoTsumura2018}.  Here we
obtain the same duality directly from the generating-function
expansion underlying Theorem~\ref{thm:explicit}.  The resulting
symmetric identity will also be used in both congruence theorems.

\begin{theorem}\label{thm:generating}
In the ring of formal power series in $x_1,\ldots,x_r,y_1,\ldots,y_r$,
\begin{equation}\label{eq:symmetric}
\begin{aligned}
&\sum_{\substack{m_1,\ldots,m_r\geq0\\k_1,\ldots,k_r\geq0}}
 {\mathbb B}_{m_1,\ldots,m_r}^{(-k_1,\ldots,-k_r)}
 \frac{x_1^{m_1}\cdots x_r^{m_r}}{m_1!\cdots m_r!}
 \frac{y_1^{k_1}\cdots y_r^{k_r}}{k_1!\cdots k_r!}
\\[2pt]
&\quad=\prod_{j=1}^r
\frac{\exp\left(\sum_{h=j}^r(x_h+y_h)\right)}
{\exp\left(\sum_{h=j}^r x_h\right)
 +\exp\left(\sum_{h=j}^r y_h\right)
 -\exp\left(\sum_{h=j}^r(x_h+y_h)\right)}.
\end{aligned}
\end{equation}
\end{theorem}

\begin{proof}
Use the defining series \eqref{eq:def} in the form obtained by
expanding the multiple polylogarithm, namely
\begin{equation}\label{eq:unreversed-geometric}
\begin{aligned}
&\sum_{m_1,\ldots,m_r\geq0}
 {\mathbb B}_{m_1,\ldots,m_r}^{(-k_1,\ldots,-k_r)}
 \frac{x_1^{m_1}\cdots x_r^{m_r}}{m_1!\cdots m_r!}
\\
&\qquad=\sum_{\ell_1,\ldots,\ell_r\geq0}
 \prod_{j=1}^r
 \left(1-e^{-\sum_{h=j}^r x_h}\right)^{\ell_j}
 Q_j(\boldsymbol\ell)^{k_j}.
\end{aligned}
\end{equation}
This is the basic expansion used in the proof of
Theorem~\ref{thm:explicit}.
Multiply by $\prod_{j=1}^r y_j^{k_j}/k_j!$ and sum over all
nonnegative upper indices.  The exponential series gives
\begin{equation}\label{eq:duality-calc}
\begin{aligned}
&\sum_{\substack{m_1,\ldots,m_r\geq0\\k_1,\ldots,k_r\geq0}}
 {\mathbb B}_{m_1,\ldots,m_r}^{(-k_1,\ldots,-k_r)}
 \frac{x_1^{m_1}\cdots x_r^{m_r}}{m_1!\cdots m_r!}
 \frac{y_1^{k_1}\cdots y_r^{k_r}}{k_1!\cdots k_r!}
\\
&=\sum_{\ell_1,\ldots,\ell_r\geq0}
\prod_{j=1}^r\left(1-e^{-\sum_{h=j}^r x_h}\right)^{\ell_j}
\exp\left(\sum_{j=1}^r Q_j(\boldsymbol\ell)y_j\right)
\\
&=\prod_{j=1}^r\left\{
\exp(jy_j)
\sum_{\ell_j=0}^{\infty}
\left[
 \left(1-e^{-\sum_{h=j}^r x_h}\right)
 \exp\left(\sum_{h=j}^r y_h\right)
\right]^{\ell_j}\right\}
\\
&=\prod_{j=1}^r
\frac{\exp(jy_j)}
{1-\left(1-e^{-\sum_{h=j}^r x_h}\right)
\exp\left(\sum_{h=j}^r y_h\right)}.
\end{aligned}
\end{equation}
In the second equality we used
$\sum_{j=1}^r Q_j(\boldsymbol\ell)y_j
=\sum_{j=1}^r j y_j+\sum_{j=1}^r\ell_j\sum_{h=j}^r y_h$.
The product of the numerator exponentials is also
\[
\exp\left(\sum_{j=1}^r jy_j\right)
=\prod_{j=1}^r\exp\left(\sum_{h=j}^r y_h\right).
\]
Multiplying the $j$th fraction in \eqref{eq:duality-calc} by
$\exp(\sum_{h=j}^r x_h)$ in numerator and denominator gives
\eqref{eq:symmetric}.  Every geometric series used above has a
ratio with zero constant term in the $x$-variables, so the
manipulations are valid formally.
\end{proof}

\begin{remark}
\rm{The} case of $r=1$ in Theorem \ref{thm:generating} coincides with the well-known fact given 
by Kaneko (see \cite[Proof of Theorem 2]{Kaneko}). 
\end{remark}

\begin{remark}
\rm{As} a multiple version of ${\mathbb B}_n^{(k)}$, 
the following definition is also known.
\begin{align*} 
\frac{\mathrm{Li}_{k_1, \ldots, k_r}(1 - e^{-t})}{(1 - e^{-t})^r} = \sum_{n = 0}^\infty \mathbb{B}_{n}^{(k_1, \ldots, k_r)} \frac{t^{n} }{n!},
\end{align*}
where
\begin{align*}
\mathrm{Li}_{k_1, \ldots, k_r} (z) = \sum_{0 < m_1 < \dots < m_r} \frac{z^{m_r}}{m_1^{k_1} m_2^{k_2} \dots m_r^{k_r}} 
\end{align*}
for $r \in \mathbb{N}, k_1, \ldots, k_r \in \mathbb{Z}$ and $z \in \mathbb{C}$ with $|z| \le 1$. 
Notice that by setting all $z_i=z$ ($1\leq i \leq r$) in $\mathrm{Li}_{k_1, \ldots, k_r}^{\sha} (z_1, \ldots, z_r)$, we obtain $\mathrm{Li}_{k_1, \ldots, k_r} (z)$.
That is, restricting all $x_i$ ($1\leq i\leq r-1$) in \eqref{eq:def}to $0$ yields $\mathbb{B}_{n}^{(k_1, \ldots, k_r)}$. In other words, $\mathbb{B}_{0, \ldots, 0, n}^{(k_1, \ldots, k_r)}=\mathbb{B}_{n}^{(k_1, \ldots, k_r)}$.
The properties of this Bernoulli numbers  $\mathbb{B}_{n}^{(k_1, \ldots, k_r)}$ also have been investigated well(cf. \cite{Bayad_Hamahata}, \cite{Furusho}, \cite{Hamahata_Masubuchi1}, \cite{Hamahata_Masubuchi2}, \cite{Kamano}).
Especially, Kamano obtained a generating function representation for these Bernoulli numbers of negative index(\cite[Theorem 2]{Kamano}), 
and our result of this article, Theorem \ref{thm:generating}, includes his result as a special case.
\end{remark}

\begin{corollary}\label{cor:duality}
For $m_1,\ldots,m_r,k_1,\ldots,k_r\in\mathbb Z_{\geq0}$,
\begin{equation}\label{eq:duality}
 {\mathbb B}_{m_1,\ldots,m_r}^{(-k_1,\ldots,-k_r)}
 ={\mathbb B}_{k_1,\ldots,k_r}^{(-m_1,\ldots,-m_r)}.
\end{equation}
\end{corollary}
\begin{proof}
The right-hand side of \eqref{eq:symmetric} is unchanged when
$x_j$ and $y_j$ are interchanged for every $j$.
Comparing the coefficients of
$x_1^{m_1}\cdots x_r^{m_r}y_1^{k_1}\cdots y_r^{k_r}$
gives \eqref{eq:duality}.
\end{proof}

\begin{corollary}\label{cor:integral}
For nonnegative integers $m_1,\ldots,m_r,k_1,\ldots,k_r$,
\[
 {\mathbb B}_{m_1,\ldots,m_r}^{(-k_1,\ldots,-k_r)}\in\mathbb Z.
\]
\end{corollary}
\begin{proof}
Apply Theorem~\ref{thm:explicit} with upper indices $-k_j$.
Each power $Q_j(\boldsymbol\ell)^{k_j}$, each multinomial
coefficient, and each Stirling number is an integer, and the
outer sum is finite.
\end{proof}

\section{Periodicity modulo prime powers}\label{sec:periodicity}

The main advantage of Theorem~\ref{thm:explicit} for arithmetic
applications is that, after duality, each lower index occurs only
as an exponent of an integer $Q_j(\boldsymbol\ell)$.
This permits the argument of Baba, Nakasuji, and Sakata
\cite[Theorem~7]{BabaNakasujiSakata2025} to be applied in every
depth and in each coordinate separately.

\begin{theorem}\label{thm:periodicity}
Let $r\geq1$, let $k_1,\ldots,k_r,N$ be positive integers, and
let $p$ be a prime.  Fix $j\in\{1,\ldots,r\}$.  Suppose
$n_j,m_j\geq N$ and
\begin{equation}\label{eq:period-hyp}
 n_j\equiv m_j\pmod{\varphi(p^N)}.
\end{equation}
If the other $n_i$ are arbitrary fixed nonnegative integers, then
\begin{equation}\label{eq:periodicity}
\begin{aligned}
&{\mathbb B}_{n_1,\ldots,n_{j-1},n_j,n_{j+1},\ldots,n_r}^{(-k_1,\ldots,-k_r)}
\\
&\qquad\equiv
{\mathbb B}_{n_1,\ldots,n_{j-1},m_j,n_{j+1},\ldots,n_r}^{(-k_1,\ldots,-k_r)}
\pmod{p^N}.
\end{aligned}
\end{equation}
\end{theorem}

\begin{proof}
By the duality \eqref{eq:duality}, it suffices to compare
\[
 {\mathbb B}_{k_1,\ldots,k_r}^{(-n_1,\ldots,-n_r)}
 \quad\text{and}\quad
 {\mathbb B}_{k_1,\ldots,k_r}^{(-n_1,\ldots,-n_{j-1},-m_j,-n_{j+1},\ldots,-n_r)}.
\]
Apply \eqref{eq:explicit} to both numbers.  Since the lower
indices $k_1,\ldots,k_r$ are fixed, both expressions have the same
finite range
\[
\ell_1,\ldots,\ell_r\geq0,
\qquad\ell_1+\cdots+\ell_r\leq k_1+\cdots+k_r,
\]
and the same integer factors, except for
$Q_j(\boldsymbol\ell)^{n_j}$ and
$Q_j(\boldsymbol\ell)^{m_j}$, respectively.
In particular, the factors multiplying either power are
integers, because the multinomial coefficients and the
Stirling numbers are integers.

Split the outer sum into the terms with
$p\mid Q_j(\boldsymbol\ell)$ and those with
$p\nmid Q_j(\boldsymbol\ell)$.
In the first class, the assumptions $n_j,m_j\geq N$ imply
\[
Q_j(\boldsymbol\ell)^{n_j}\equiv
Q_j(\boldsymbol\ell)^{m_j}\equiv0\pmod{p^N}.
\]
In the second class, Euler's theorem gives
$Q_j(\boldsymbol\ell)^{\varphi(p^N)}\equiv1\pmod{p^N}$.
Together with \eqref{eq:period-hyp}, this implies
\[
Q_j(\boldsymbol\ell)^{n_j}\equiv
Q_j(\boldsymbol\ell)^{m_j}\pmod{p^N}.
\]
Thus the corresponding summands are congruent in both cases.
Summation and the duality formula yield \eqref{eq:periodicity}.
\end{proof}

\begin{corollary}\label{cor:all-periodicity}
Under the assumptions of Theorem~\ref{thm:periodicity}, suppose
that, for every $1\leq j\leq r$,
$n_j,m_j\geq N$ and
$n_j\equiv m_j\pmod{\varphi(p^N)}$.  Then
\begin{equation}\label{eq:all-periodicity}
 {\mathbb B}_{n_1,\ldots,n_r}^{(-k_1,\ldots,-k_r)}
 \equiv {\mathbb B}_{m_1,\ldots,m_r}^{(-k_1,\ldots,-k_r)}
 \pmod{p^N}.
\end{equation}
\end{corollary}
\begin{proof}
Apply Theorem~\ref{thm:periodicity} successively for
$j=1,\ldots,r$.
\end{proof}

\section{Congruences for sums over one period}\label{sec:sums}

The analogue of Theorem~8 of Baba, Nakasuji, and Sakata
\cite{BabaNakasujiSakata2025} concerns a sum in the first lower
index.  In addition to Euler's theorem, its proof requires a
simple divisibility lemma for geometric sums.  This takes care
of terms for which the denominator of the usual expression for
the geometric sum is divisible by $p$.

\begin{lemma}\label{lem:geometric}
Let $p$ be a prime, $N\geq1$, and let $q$ be an integer with
$q\equiv1\pmod p$.  Then
\begin{equation}\label{eq:geometric-lemma}
\sum_{i=0}^{\varphi(p^N)-1}q^i\equiv0\pmod{p^{N-1}}.
\end{equation}
For $N=1$, the asserted congruence is understood modulo $1$.
\end{lemma}
\begin{proof}
For $N=1$ there is nothing to prove.  For $N\geq2$,
\[
 \sum_{i=0}^{p^{N-1}-1}q^i
 =\prod_{d=0}^{N-2}
 \left(1+q^{p^d}+\cdots+q^{(p-1)p^d}\right).
\]
Since $q\equiv1\pmod p$, each factor on the right is divisible
by $p$.  Hence the product is divisible by $p^{N-1}$.
Finally, $\varphi(p^N)=p^{N-1}(p-1)$, and the longer sum in
\eqref{eq:geometric-lemma} is a sum of $p-1$ shifted copies of
this product.  The claim follows.
\end{proof}

\begin{theorem}\label{thm:period-sum}
Let $r\geq1$, let $k_1,\ldots,k_r,N$ be positive integers,
and let $p$ be a prime.  If $n_1\geq N$ and
$n_2,\ldots,n_r\geq0$, then
\begin{equation}\label{eq:period-sum}
 \sum_{i=0}^{\varphi(p^N)-1}
 {\mathbb B}_{n_1+i,n_2,\ldots,n_r}^{(-k_1,\ldots,-k_r)}
 \equiv0\pmod{p^N}.
\end{equation}
\end{theorem}

\begin{proof}
As in the proof of Theorem~8 of
\cite{BabaNakasujiSakata2025}, first consider $n_1=N$.
Use duality \eqref{eq:duality} and the explicit formula
\eqref{eq:explicit} to expand the sum on the left.  For each
fixed tuple $\boldsymbol\ell$, the only part depending on the
summation index $i$ is $Q_1(\boldsymbol\ell)^{N+i}$.
Thus, after summing over $i$, every contribution is an integer
multiple of
\begin{equation}\label{eq:crucial-factor}
 \ell_1!\,Q_1(\boldsymbol\ell)^N
 \sum_{i=0}^{\varphi(p^N)-1}Q_1(\boldsymbol\ell)^i.
\end{equation}
To see that the terms with $\ell_1=0$ do not contribute, note
that the first Stirling-number factor in \eqref{eq:explicit},
with $m_j=k_j$, is
\[
 \Stir{k_1+n_{1,1}+\cdots+n_{r-1,1}}{\ell_1}.
\]
Since $k_1\geq1$, this factor vanishes when $\ell_1=0$.
The same conclusion holds for $r=1$, when the factor is
$\Stir{k_1}{\ell_1}$.  Therefore we may assume
$\ell_1\geq1$ in \eqref{eq:crucial-factor}.

We now distinguish three cases.

\smallskip\noindent
\emph{Case 1: $p\mid Q_1(\boldsymbol\ell)$.}
Since $Q_1(\boldsymbol\ell)=\ell_1+1$ and $N\geq1$,
$Q_1(\boldsymbol\ell)^N$ is divisible by $p^N$.
Hence \eqref{eq:crucial-factor} vanishes modulo $p^N$.

\smallskip\noindent
\emph{Case 2: $p\nmid Q_1(\boldsymbol\ell)$ and
$p\nmid\ell_1$.}
In this case $Q_1(\boldsymbol\ell)-1=\ell_1$ is a unit
modulo $p^N$.  The geometric-sum identity and Euler's theorem
give
\[
\sum_{i=0}^{\varphi(p^N)-1}Q_1(\boldsymbol\ell)^i
=\frac{Q_1(\boldsymbol\ell)^{\varphi(p^N)}-1}
       {Q_1(\boldsymbol\ell)-1}
\equiv0\pmod{p^N}.
\]
Thus \eqref{eq:crucial-factor} again vanishes modulo $p^N$.

\smallskip\noindent
\emph{Case 3: $p\mid\ell_1$.}
Because $\ell_1\geq1$, the assumption $p\mid\ell_1$ implies
$\ell_1\geq p$ and therefore $p\mid\ell_1!$.
Also $Q_1(\boldsymbol\ell)=\ell_1+1\equiv1\pmod p$.
By Lemma~\ref{lem:geometric},
\[
 \sum_{i=0}^{\varphi(p^N)-1}Q_1(\boldsymbol\ell)^i
 \equiv0\pmod{p^{N-1}}.
\]
The factorial $\ell_1!$ supplies the remaining factor $p$,
and \eqref{eq:crucial-factor} vanishes modulo $p^N$.

All terms in the explicit expansion therefore vanish modulo
$p^N$, proving \eqref{eq:period-sum} for $n_1=N$.

For arbitrary $n_1\geq N$, Theorem~\ref{thm:periodicity} shows
that the sequence
${\mathbb B}_{n_1,n_2,\ldots,n_r}^{(-k_1,\ldots,-k_r)}$
is periodic modulo $p^N$ with period $\varphi(p^N)$ as
$n_1$ ranges over the integers $n_1\geq N$.
Each interval of $\varphi(p^N)$ consecutive indices contains
one representative of every residue class modulo this period.
Consequently,
\[
\sum_{i=0}^{\varphi(p^N)-1}
 {\mathbb B}_{n_1+i,n_2,\ldots,n_r}^{(-k_1,\ldots,-k_r)}
\equiv
\sum_{i=0}^{\varphi(p^N)-1}
 {\mathbb B}_{N+i,n_2,\ldots,n_r}^{(-k_1,\ldots,-k_r)}
\equiv0\pmod{p^N}.
\]
This completes the proof.
\end{proof}

\begin{corollary}\label{cor:general-mod}
Let $r\geq1$, $k_1,\ldots,k_r\geq1$,
$n_2,\ldots,n_r\geq0$, and let
$M=\prod_{a=1}^{s}p_a^{e_a}$, where the $p_a$ are distinct
primes and $e_a\geq1$.  If
$n_1\geq\max\{e_1,\ldots,e_s\}$, then
\begin{equation}\label{eq:general-mod}
 \sum_{i=0}^{\varphi(M)-1}
 {\mathbb B}_{n_1+i,n_2,\ldots,n_r}^{(-k_1,\ldots,-k_r)}
 \equiv0\pmod M.
\end{equation}
\end{corollary}
\begin{proof}
Fix a prime-power factor $p_a^{e_a}$ of $M$.
Since $\varphi(p_a^{e_a})$ divides $\varphi(M)$, the sum
in \eqref{eq:general-mod} splits into complete blocks of
$\varphi(p_a^{e_a})$ consecutive terms.  Each block is zero
modulo $p_a^{e_a}$ by Theorem~\ref{thm:period-sum}, because
every starting index is at least $n_1\geq e_a$.
Therefore \eqref{eq:general-mod} holds modulo each
$p_a^{e_a}$.  The Chinese remainder theorem completes the proof.
\end{proof}

\begin{remark}\label{rem:asymmetry}
The coordinatewise periodicity in Theorem~\ref{thm:periodicity}
holds for every $j$, but the same is not true of the period-sum
statement.  For example, setting $k_1=k_2=1$ in
\eqref{eq:double} and applying duality gives
\begin{equation}\label{eq:counterformula}
 {\mathbb B}_{n_1,n_2}^{(-1,-1)}
 =2^{n_1}4^{n_2}-2^{n_1}3^{n_2}
  +2\cdot3^{n_1}4^{n_2}.
\end{equation}
For $p=3$, $N=1$, and $n_1=n_2=1$, this yields
\[
 \sum_{i=0}^{\varphi(3)-1}{\mathbb B}_{1,1+i}^{(-1,-1)}
 ={\mathbb B}_{1,1}^{(-1,-1)}+{\mathbb B}_{1,2}^{(-1,-1)}
 =26+110=136\equiv1\pmod3.
\]
Thus Theorem~\ref{thm:period-sum} cannot be extended by
simply replacing the first lower index with the second.
This illustrates the issue raised in
\cite[Remark~1 following Theorem~8]{BabaNakasujiSakata2025}:
the first variable has the special factor
$Q_1(\boldsymbol\ell)-1=\ell_1$, whose possible
$p$-divisibility is compensated by the factor $\ell_1!$.
\end{remark}

\section*{Acknowledgement}
The authors would like to thank Professor Hirofumi Tsumura and Professor Mika Sakata for their helpful comments.
This work was supported by Japan Society for the Promotion of Science, Grant-in-Aid for Scientific Research (C) No. 22K03274 (M. Nakasuji).

\end{document}